\documentclass[12pt, reqno]{amsart}
\usepackage{amsmath, amsthm, amscd, amsfonts, amssymb, graphicx, color}
\usepackage[bookmarksnumbered, colorlinks, plainpages]{hyperref}

\newtheorem{theorem}{Theorem}[section]

\theoremstyle{definition}
\newtheorem{definition}[theorem]{Definition}
\newtheorem{example}[theorem]{Example}

\theoremstyle{remark}
\newtheorem{remark}[theorem]{Remark}
\numberwithin{equation}{section}

\begin{document}
\setcounter{page}{1}

\title[Indefinite Kenmotsu statistical manifold]{Contact $CR$ and $SCR$ lightlike submanifolds of an indefinite Kenmotsu statistical manifold}

\author[Shagun, J. Kaur]{Shagun$^1$, Jasleen Kaur$^2$$^{*}$}

\address{$^{1}$ Department of Mathematics, Punjabi University, Patiala, India.}
\email{shagun.bhatti82@gmail.com}

\address{$^{2}$ Department of Mathematics, Punjabi University, Patiala, India.}
\email{jass2381@gmail.com}


\subjclass[2010]{Primary 58A05; Secondary 53D10, 53D15.}

\keywords{Indefinite Kenmotsu manifold, indefinite statistical structure, totally geodesic foliation, integrability.}

\begin{abstract}
This research article introduces the concept of lightlike submanifolds of an indefinite Kenmotsu statistical manifold. Various results on geometry of contact $CR$ and $SCR$-lightlike submanifolds have been developed. Some characterization theorems on the integrability of distributions and on the geodesicity of these submanifolds have been obtained. Appropriate examples have been presented.
\end{abstract} \maketitle

\section{Introduction}

\noindent Statistical manifolds, which are geometrically formulated as Riemannian manifolds with a certain affine connection were introduced by \cite{ref12} and afterwards explored by \cite{ref1}, \cite{ref6}, \cite{ref13}, et.al., where they studied the probability information from the perspective of differential geometry. Henceforth, the notion of statistical counterpart of the Kenmotsu manifold, that is, Kenmotsu statistical manifold was initiated by \cite{ref14} and subseqently researched by \cite{ref2}, \cite{ref3}, \cite{ref4} et.al. for its geometrical properties and statistical curvature. The statistical manifolds hold a considerable importance in the realm of neural networks, control systems and statistical inference.\\

\noindent Thus motivated, we introduce the idea of indefinite Kenmotsu statistical manifold by consolidating the concept of statistical structure with the indefinite Kenmotsu metric structure and examine its lightlike theory of submanifolds.The geometry of contact $CR$ and $SCR$-lightlike submanifolds of the indefinite Kenmotsu statistical manifold has been detailed with examples. The characterization theorems on mixed geodesicity, totally geodesicity, integrability of various distributions and on totally geodesic foliations in contact  $CR$ and $SCR$ lightlike submanifolds have been given.

\section{Lightlike geometry of an indefinite statistical manifold}

\noindent The lightlike theory of submanifolds following \cite{ref5}, \cite{ref8} is as follows:\\

\noindent Let $M$ be a $m$-dimensional submanifold of a real $(m+n)$-dimensional semi-Riemannian manifold $(\bar{M},\bar{g})$ where $m,n \geq 1$. If $M$ admits a degenerate metric $g$ induced from $\bar{g}$, whose Radical distribution $Rad(TM)$ of rank $r$ ($1 \leq r \leq m$) is defined by
\[
Rad (T_{x}M) = \{\xi \in T_{x}M : g_{x}(\xi, X) = 0,\; X \in \Gamma(TM),\; x \in M\}
\]
such that $Rad(TM) = TM \cap TM^{\perp}$, where
\[
TM^{\perp} = \bigcup_{x \in M} \{u \in T_{x}\bar{M} : \bar{g}(u, v) = 0, \; \forall \; v \in T_{x}M\},
\]
then $M$ is said to be lightlike submanifold of $\bar{M}$.\\

\noindent Also, $TM = Rad(TM) \perp S(TM)$ where $S(TM)$ is a complementary distribution of $Rad(TM)$ in $TM$ known as screen distribution and $TM^{\perp} = Rad(TM) \perp S(TM^{\perp})$ where $S(TM^{\perp})$ is a complementry vector bundle of  $Rad(TM)$ in $TM^{\perp}$ known as screen transversal vector bundle.\\

\noindent Furthermore, there exists a local frame $\{N_{i}\}$ of sections with values in the orthogonal complement of $S(TM^{\perp})$ in $(S(TM))^{\perp}$ such that $\bar{g}(\xi_{i},N_{j}) = \delta_{ij}$ and $\bar{g}(N_{i},N_{j}) = 0$ for any local basis $\{\xi_{i}\}$ of $Rad(TM)$. It follows that there exists a lightlike transversal vector bundle $ltr(TM)$ locally spanned by $\{N_{i}\}$ where $i=1,2,...,r$.\\

\noindent Let $tr(TM)$ be the complementary (but not orthogonal) vector bundle to $TM$ in $T\bar{M}\mid_{M}$.\\

\noindent Then
\[
tr(TM) = ltr(TM) \perp S(TM^{\perp}),
\]
\[
T\bar{M}\mid_{M} = S(TM) \perp \{Rad(TM) \oplus ltr(TM)\} \perp S(TM^{\perp}).
\]

\noindent Let $\tilde{\nabla}$ be the Levi-Civita connection on $\bar{M}$, then the Gauss and Weingarten formulae are given by
\begin{equation}\label{27}
	\tilde{\nabla}_{X}Y = \nabla_{X}Y + h(X, Y ), \hspace{.2in} \forall \; X, Y \in\Gamma(TM)
\end{equation}
\begin{equation}\label{28}
	\tilde{\nabla}_{X}V = - A_{V}X + \nabla_{X}^{\perp}V, \hspace{.2in} \forall \;  X \in\Gamma(TM), \; V \in \Gamma(tr(TM))
\end{equation}
where $\nabla_{X}Y,\; A_{V}X \in \Gamma(TM)$ and $h(X, Y), \; \nabla_{X}^{\perp}V \in \Gamma(tr(TM))$. Here, $\nabla$ and $\nabla^{\perp}$ are the linear connections on $M$ and vector bundle $tr(TM)$, respectively.\\

\noindent 	Consider the projection morphisms
\[
L: tr(TM) \rightarrow ltr(TM) \quad and \quad S: tr(TM) \rightarrow S(TM^{\perp})
\]
then, (\ref{27}) and (\ref{28}) become
\begin{equation}\label{eq29}
	\begin{split}
		\tilde{\nabla}_{X}Y = \nabla_{X}Y + h^{l}(X, Y ) + h^{s}(X, Y ), \\
		\tilde{\nabla}_{X}V = - A_{V}X + D_{X}^{l}V + D_{X}^{s}V.
	\end{split}
\end{equation}

\noindent In particular, we have
\begin{equation}\label{eq210}
	\begin{split}
		\tilde{\nabla}_{X}N = - A_{N}X + \nabla_{X}^{l}N + D^{s}(X, N), \\ \tilde{\nabla}_{X}W = - A_{W}X + \nabla_{X}^{s}W + D^{l}(X, W)
	\end{split}
\end{equation}

\noindent where $X, Y \in \Gamma(TM), \; N \in \Gamma(ltr(TM))$ and $W \in \Gamma(S(TM^{\perp}))$.\\

\noindent	Here, $\nabla_X^l N, \; D^l(X,W) \in \Gamma(ltr(TM)),\quad \nabla_X^s W, \; D^s(X,N)  \in \Gamma(S(TM^{\perp})),\\ \quad \nabla_X Y,\; A_N X, \; A_W X \in \Gamma(TM)$ and $h^l(X,Y) = Lh(X,Y), h^s(X,Y) = Sh(X,Y),\\ D_X^{l}V=L( \nabla_X^{\perp}V), \; D_X^{s}V=S( \nabla_X^{\perp}V)$; wherein $h^{l}$ and $h^{s}$ are respectively known as the lightlike second fundamental form and the screen second fundamental form on $M$. \\

\noindent Let $P$ be the projection morphism of $TM$ on $S(TM)$, then we consider the following decompositions:
\[
\nabla_{X}PY = \nabla_{X}^{\prime}PY + h^{\prime}(X, PY), \hspace{.2in} \nabla_{X}\xi= - A_{\xi}^{\prime}X + {\nabla_{X}^{\prime}}^{\perp}\xi
\]
where $\{\nabla_{X}^{\prime}Y, - A_{\xi}^{\prime}X\}$ $\in$ $\Gamma(S(TM))$ and $\{h^{\prime}(X, PY), {\nabla_{X}^{\prime}}^{\perp}\xi\}$ $\in$ $\Gamma(Rad(TM))$, respectively; $\nabla^{\prime}$ and ${\nabla^{\prime}}^{\perp}$ are linear connections on $S(TM)$ and $Rad(TM)$, respectively.\\

\noindent From the above equations, we obtain
\[
\bar{g}(h^{l}(X, PY), \xi) = g(A_{\xi}^{\prime}X, PY), \hspace{.2in} \bar{g}(h^{\prime}(X, PY), N) = g(A_{N}X, PY),
\]
\[
\bar{g}(h^{l}(X, \xi), \xi) = 0, \hspace{.2in} g(A_{\xi}^{\prime}PX, PY)=g(PX, A_{\xi}^{\prime}PY),
\]
\[
\bar{g}(A_{\xi}^{\prime}X, N) = 0, \hspace{.2in} \bar{g}(A_{N}X, N) = 0, \hspace{.2in} A_{\xi}^{\prime}\xi = 0 
\]

\noindent	for any $X, Y \in \Gamma(TM),\; \xi \in \Gamma(Rad(TM))$ and $N \in \Gamma(ltr(TM))$.

\begin{definition}
	A pair $(\bar{\nabla}, \bar{g})$ is called an indefinite statistical structure on a semi-Riemannian manifold $\bar{M}$ where $\bar{g}$ is a semi-Riemannian metric of constant index $q\geq1$ on $\bar{M}$, if 
	\[
	\bar{\nabla}_{X}Y-\bar{\nabla}_{Y}X=[X,Y]
	\]
	\[
	\textnormal{and} \quad (\bar{\nabla}_{X}\bar{g})(Y, Z) = (\bar{\nabla}_{Y}\bar{g})(X, Z)
	\]
	hold for any $X, Y, Z \in \Gamma(T\bar{M})$.
	
\end{definition}

\noindent Moreover, there exists $\bar{\nabla}^{*}$ which is a dual connection of $\bar{\nabla}$ with respect to $\bar{g}$, such that
\begin{equation}\label{eq37}
	X\bar{g}(Y, Z) = \bar{g}(\bar{\nabla}_{X}Y, Z) + \bar{g}(Y, \bar{\nabla}_{X}^{*}Z) \qquad \forall \; X, Y, Z \in \Gamma(T\bar{M}).
\end{equation}

\noindent If $(\bar{\nabla},\bar{g})$ is an indefinite statistical structure on $\bar{M}$, then so is $(\bar{\nabla}^{*}, \bar{g})$. Hence, the indefinite statistical manifold is denoted by $(\bar{M},\bar{\nabla},\bar{\nabla}^{*},\bar{g})$.\\

\noindent The lightlike theory available so far leads to the Gauss and Weingarten formulae for the structure of a lightlike submanifold $(M,g)$ of an indefinite statistical manifold  $(\bar{M},\bar{g},\bar{\nabla},\bar{\nabla^*})$  as follows:
\begin{equation}\label{eq38}
	\begin{split}
		\bar\nabla_X Y = \nabla_X Y +h^l(X,Y) + h^s(X,Y),  \\
		\bar\nabla^*_X Y = \nabla^*_X Y +h^{*l}(X,Y) +h^{*s}(X,Y),   
\end{split} \end{equation}	
\[
\bar\nabla_X V = -A_V X + D_X^{l}V +D_X^{s}V , \quad \bar\nabla_X^* V = -A_V^* X + D_X^{*l}V +D_X^{*s}V ,
\]
\[
\bar\nabla_X N = -A_N X + \nabla^l_X N +D^s(X,N), \quad \bar\nabla_X^* N = -A_N^* X + \nabla^{*l}_X N +D^{*s}(X,N),
\]
\begin{equation}\label{eqF}
	\begin{split}
		\bar\nabla_X W =-A_W X +\nabla^s_X W +D^l(X,W),\\ \bar\nabla_X^*W =-A_W^*X + \nabla^{*s}_XW +D^{*l}(X,W)
	\end{split}
\end{equation}

for any  $X, Y \in \Gamma(TM)$, $V \in \Gamma(tr(TM))$, $N \in \Gamma(ltr(TM))$ and $W\in \Gamma(STM^{\perp})$.\\

\noindent 	Considering the corresponding projection morphism $P$ of tangent bundle $TM$ to the screen distribution, the following decompositions w.r.t $\nabla$ and $\nabla^*$ hold:
\[
\nabla_X PY = \nabla_X^{\prime} PY + h^{\prime}(X,PY), \quad \nabla_X^* PY = \nabla_X^{*\prime} PY + h^{*\prime}(X,PY),
\]
\begin{equation}\label{eq312}
	\nabla_X \xi = - A^{\prime}_{\xi} X + \nabla_X^{\prime \perp} \xi, \quad \nabla_X^* \xi = - A^{*\prime}_{\xi} X + {\nabla^*}_X^{\prime \perp} \xi
\end{equation}
for any $X, Y \in \Gamma(TM)$, $\xi \in \Gamma(Rad(TM))$. \\

\noindent 	Thus, we have
\[
\bar g(h^l(X,PY),\xi) = g(A^{*\prime}_\xi X,PY), \quad \bar g(h^{*l}(X,PY),\xi) = g(A^{\prime}_\xi X,PY),
\]
\[
\bar g(h^{\prime}(X,PY),N) = g(A^*_NX,PY), \quad \bar g(h^{*\prime}(X,PY),N) = g(A_NX,PY).
\]

\section{Indefinite Kenmotsu statistical manifold}

\noindent Let $\widehat{\bar{\nabla}}$ be the Levi-Civita connection of $\bar{g}$. Therefore, by \cite{ref7}, we consider the decomposition $\widehat{\bar{\nabla}} = \frac{1}{2}(\bar{\nabla} + \bar{\nabla}^{*})$, where $\bar{\nabla}^{*}$ is called the dual connection of $\bar{\nabla}$ with respect to $\bar{g}$.\\

\noindent 	For a statistical structure $(\bar{\nabla}, \bar{g})$ on $\bar{M}$, we set 
\begin{equation}\label{eqR1}
	K_{X}Y =K(X, Y)= \bar{\nabla}_{X}Y - \widehat{\bar{\nabla}}_{X}Y
\end{equation}
for any $X, Y \in \Gamma(TM)$ where $K$ is the difference (1,2) type tensor of a torsion free affine connection $\bar{\nabla}$ and the Levi-Civita connection $\widehat{\bar{\nabla}}$.\\

\noindent Since $\bar{\nabla}$ and $\widehat{\bar{\nabla}}$ are torsion free, then we have 
\begin{equation}\label{eq316}
	K_{X}Y = K_{Y}X, \quad \bar{g}(K_{X}Y, Z) = \bar{g}(Y, K_{X}Z) \qquad \forall \; X,Y,Z \in \Gamma(T\bar{M}).
\end{equation} 	

\noindent Also, 
\begin{equation}\label{eqR2}
	K(X,Y) = \widehat{\bar{\nabla}}_{X}Y - \bar{\nabla}^{*}_{X}Y = \frac{1}{2}(\bar{\nabla}_{X}Y - \bar{\nabla}^{*}_{X}Y).
\end{equation}

\noindent Let $(\bar{M},\bar{g})$ be a semi-Riemannian manifold of dimension $(2n+1)$. If $\bar{g}$ is a semi-Riemannian metric, $\phi$ is a $(1, 1)$ tensor field, $\nu$ is a characteristic vector field and $\eta$ is a $1$-form, such that
\begin{equation}\label{eq31}
	\bar{g}(\phi X, \phi Y) = \bar{g}(X, Y) - \eta(X)\eta(Y), \qquad \bar{g}(\nu, \nu) = 1,   
\end{equation}
\begin{equation}\label{eq32}
	\phi^{2}(X) = - X + \eta(X)\nu, \quad \bar{g}(X, \nu) = \eta(X),\quad \bar{g}(\phi X, Y)  + \bar{g}(X, \phi Y) = 0
\end{equation}

\noindent	which follows that $\phi\nu = 0$ and  $\eta o \phi = 0$ for all $X,Y \in \Gamma(TM)$, then $(\phi, \nu, \bar{g})$ is called an almost contact metric structure on $\bar{M}$.

\begin{definition}	
	\cite{ref9} An almost contact metric structure on $\bar{M}$ is called an indefinite Kenmotsu structure if
	\begin{equation}\label{eq33}
		(\widehat{\bar{\nabla}}_{X}\phi)Y=\eta(Y)\phi X-\bar{g}(\phi X, Y)\nu, \quad 	\widehat{\bar{\nabla}}_{X}\nu=-X+\eta(X)\nu	
	\end{equation}
	
	\noindent	holds for any $X, Y \in \Gamma(TM)$, where $\widehat{\bar{\nabla}}$ is a Levi-Civita Connection.  
\end{definition}	

\begin{definition}\label{def1}
	A quadruplet $(\bar{\nabla} = \widehat{\bar{\nabla}} + K, \bar{g}, \phi, \nu)$ is called an indefinite Kenmotsu statistical structure on $\bar{M}$ if 
	\begin{enumerate}
		\item $(\bar{g}, \phi, \nu)$ is an indefinite Kenmotsu structure,
		\item $(\bar{\nabla}, \bar{g})$ is a statistical structure on $\bar{M}$ and
	\end{enumerate}
	\begin{equation}\label{eq317}
		K(X, \phi Y) = - \phi K(X, Y) 
	\end{equation}
	holds for any $X, Y \in \Gamma(TM)$. Then $(\bar{M}, \bar{\nabla}, \bar{g},\phi, \nu)$ is said to be an indefinite Kenmotsu statistical manifold.
\end{definition}

\noindent If $(\bar{M}, \bar{\nabla}, \bar{g}, \phi, \nu)$ is an indefinite Kenmotsu statistical manifold, so is $(\bar{M}, \bar{\nabla}^{*}, \bar{g}, \phi, \nu)$.	

\begin{theorem}
	For an almost contact metric structure $(\bar{g}, \phi, \nu)$ on an indefinite statistical manifold $(\bar{M}, \bar{\nabla}, \bar{\nabla}^{*}, \bar{g})$, $(\bar{\nabla}, \bar{\nabla}^{*}, \bar{g}, \phi, \nu)$ is said to be an indefinite Kenmotsu statistical struture on $\bar{M}$ iff:
	\begin{equation}\label{eq67}
		\bar{\nabla}_{X}\phi Y - \phi\bar{\nabla}^{*}_{X}Y = \eta(Y)\phi X - \bar{g}(\phi X, Y)\nu \quad and
	\end{equation}
	\begin{equation}\label{eq68}
		\bar{\nabla}_{X}\nu = -X + \{\eta(X)+\mu(X)\}\nu
	\end{equation}	
	hold for all the vector fields $X, Y$ on $\bar{M}$, where $\mu(X)= \eta(\bar{\nabla}_{X}\nu)=-\eta(\bar{\nabla}_{X}^{*}\nu)=\eta(K(\nu,\nu))\eta(X)$.
\end{theorem}

\begin{proof}
	For an indefinite Kenmotsu statistical manifold, using (\ref{eqR1}) and (\ref{eqR2}), we get
	\[
	\bar{\nabla}_{X}\phi Y-\phi \bar{\nabla}^{*}_{X}Y=K_X\phi Y+\widehat{\bar{\nabla}}_{X}\phi Y-\phi \widehat{\bar{\nabla}}_{X}Y+\phi K_XY
	\]
	\[
	\bar{\nabla}_{X}\phi Y-\phi \bar{\nabla}^{*}_{X}Y=\eta(Y)\phi X-\bar{g}(\phi X, Y)\nu.	
	\]
	
	\noindent Replacing $\bar{\nabla}$ by $\bar{\nabla}^{*}$ and $Y$ by $\nu$, we have
	\[
	\phi\bar{\nabla}_{X}\nu=-\phi X+\bar{g}(\phi X,\nu)\nu.
	\]
	Therefore, from the concept of almost contact metric structure, we obtain (\ref{eq68}).\\
	
	\noindent Conversely, replacing $Y$ by $\phi Y$ in (\ref{eq67}), we derive
	\[
	\phi\{\bar{\nabla}_{X}\phi^{2}Y-\phi \bar{\nabla}^{*}_{X}\phi Y\}=0, \; \textnormal{where} \; \eta(\phi Y)=0 \; \textnormal{and} \; \phi\nu=0.
	\]
	
	\noindent Further, using (\ref{eq37}), (\ref{eq32}) and (\ref{eq68}), above equation becomes :
	\[0=-\phi\bar{\nabla}_{X}Y+\eta(Y)\phi\bar{\nabla}_{X}\nu+\bar{\nabla}^{*}_{X}\phi Y+\bar{g}(\phi Y,\bar{\nabla}_{X}\nu)\nu\]
	\[=-\phi\bar{\nabla}_{X}Y-\eta(Y)\phi X+\bar{\nabla}^{*}_{X}\phi Y+\bar{g}(\phi X,Y)\nu\]
	
	\noindent which implies \begin{equation}\label{eq69}
		\bar{\nabla}^{*}_{X}\phi Y - \phi\bar{\nabla}_{X}Y = \eta(Y)\phi X - \bar{g}(\phi X, Y)\nu.
	\end{equation}
	
	\noindent Hence, using (\ref{eq67}) and (\ref{eq69}) respectively, we get
	\[
	(\widehat{\bar{\nabla}}_{X}\phi) Y-\eta(Y)\phi X+\bar{g}(\phi X,Y)\nu=K_{X}\phi Y+\phi K_XY 
	\] and
	\[
	(\widehat{\bar{\nabla}}_{X}\phi) Y-\eta(Y)\phi X+\bar{g}(\phi X,Y)\nu=-K_{X}\phi Y-\phi K_XY. \] 
	
\end{proof}

\begin{remark}\label{Re1}
	\cite{ref7} Let $(\bar{g}, \phi, \nu)$ be an indefinite Kenmotsu structure on $\bar{M}$. By setting
	\[
	K(X, Y) = g(X, \nu )g(Y, \nu )\nu
	\]
	for any $X, Y \in \Gamma(T\bar{M})$ such that $K$ satisfies (\ref{eq316}) and (\ref{eq317}), we obtain an indefinite Kenmotsu statistical structure $(\bar{\nabla}^{\lambda} = \widehat{\bar{\nabla}} + \lambda K, \bar{g}, \phi, \nu)$ on $\bar{M}$ 	for $\lambda \in C^{\infty}(\bar{M})$.
\end{remark}

\section{Geodesic contact CR-lightlike submanifolds}

\begin{definition}
	
	\cite{ref9} Let $(M, g, S(TM), S(TM^{\perp}))$ be a lightlike submanifold tangent to the structure vector field $\nu$, immersed in an indefinite Kenmotsu manifold $(\bar{M},\bar{g})$. Then, $M$ is a contact $CR$-lightlike submanifold of $\bar{M}$ if the following conditions are satisfied:
	
\end{definition}

\begin{enumerate}
	\item $RadTM$ is a distribution on $M$ such that $Rad(TM) \cap \phi(Rad(TM)) = \{0\},$
	\item there exist vector bundles $D_{0}$ and $D^{\prime}$ over $M$ such that
	\[
	S(TM) = \{\phi(Rad(TM)) \oplus D^{\prime}\} \perp D_{o} \perp \{\nu\}, 
	\]
	\[
	\phi D_{o} = D_{o}, \hspace{.2in} \phi(D^{\prime}) = L_{1} \perp ltr(TM)
	\]
\end{enumerate}

\noindent where $D_{o}$ is non-degenerate and $L_{1}$ is a vector subbundle of $S(TM^{\perp})$.\\

\noindent Thus, one has the following decomposition:
\begin{equation}\label{eq42b}
	TM = D \oplus D^{\prime} \perp \{\nu\}, \hspace{.2in} D = Rad(TM) \perp \phi(Rad(TM)) \perp D_{o}. 
\end{equation}

\noindent Also, a contact $CR$-lightlike submanifold is proper if $D_o \ne \{0\}$ and $L_1\ne \{0\}$. \\

\noindent Let, the orthogonal complement subbundle to the vector subbundle $L_{1}$ in $S(TM^{\perp})$ be denoted by $L_{1}^{\perp}$. For a contact $CR$-lightlike submanifold $M$, we put
\begin{equation}\label{eq45b}
	\phi X = fX + wX, \quad \forall\; X \in \Gamma(TM),
\end{equation}
where $fX \in \Gamma(D)$ and $wX \in \Gamma(L_{1} \perp ltr(TM))$. Similarly, we have
\begin{equation}\label{eq46b}
	\phi W = BW + CW, \quad \forall\; W \in \Gamma(S(TM^{\perp})),
\end{equation}
where $BW \in \Gamma(\phi L_{1})$ and $CW \in \Gamma(L_{1}^{\perp})$.\\ 

\noindent 	Inspired by \cite{ref9}, the geometry of contact $CR$-lightlike submanifold of an indefinite Kenmotsu statistical manifold is detailed with an example as follows.

\begin{example}\label{example}
	Consider a semi-Euclidean space $\bar{M}=(\mathbb{R}^{9}_{2},\bar{g})$, with canonical basis 
	\[
	\{\partial x_{1}, \partial x_{2}, \partial x_{3}, \partial x_{4}, \partial y_{1}, \partial y_{2}, \partial y_{3}, \partial y_{4}, \partial z\}
	\]
	where $\bar{g}$ is of signature $(-,+,+,+,-,+,+,+,+)$.  \\
	
	\noindent The degenerate structure of an indefinite Kenmotsu manifold $(\mathbb{R}^{9}_{2},\phi_{o}, \nu,\eta,\bar{g})$ for cartesian coordinates $(x^{i};y^{i};z)$ is given as:
	\[		
	\eta=dz, \qquad \nu=\partial z,
	\]
	\[
	\bar{g} = \eta \otimes \eta+e^{2z}(-dx^{1}\otimes dx^{1}-dy^{1}\otimes dy^{1}+dx^{3}\otimes dx^{3}+dy^{3}\otimes dy^{3}+dx^{4}\otimes dx^{4}+dy^{4}\otimes dy^{4}),
	\]
	\[
	\phi_{o}(\sum_{i=1}^{4}(X_{i}\partial x^{i}+Y_{i}\partial y^{i} )+Z\partial z)=\sum_{i=1}^{4}(Y_{i}\partial x^{i}-X_{i}\partial y_{i}).
	\]
	
	\noindent If we set $K(X, Y) = g(X, \nu )g(Y, \nu )\nu$ with $\lambda=1$, then $(\bar{\nabla} = \widehat{\bar{\nabla}} + K, \phi_{o}, \nu, \eta, \bar{g})$ defines an indefinite Kenmotsu statistical structure on $\bar{M}$ by Remark (\ref{Re1}).\\
	
	\noindent Let $M$ be a submanifold of $\mathbb{R}^{9}_{2}$ defined by
	\[
	x^{1}=y^{4}, \quad x^{2}=\sqrt{1-(y^{2})^{2}}, \quad y^{2} \neq \pm 1,
	\]
	
	\noindent Thus, the local frame of $TM$ is as below:
	\begin{equation}\label{eq43x}
		l_{1}=e^{-z}(\partial x_{1}+\partial y_{4}), \quad l_{2}=e^{-z}(\partial x_{4}-\partial y_{1}),
	\end{equation}
	\begin{equation}\label{eq44y}
		l_{3}=e^{-z}\partial x_{3}, \quad l_{4}=e^{-z}\partial y_{3}, \quad l_{5}=e^{-z}(-\frac{y^{2}}{x^{2}}\partial x_{2}+ \partial y_{2}),
	\end{equation}
	\begin{equation}\label{eq45z}
		l_{6}=e^{-z}(\partial x_{4}+\partial y_{1}), \quad l_{7}=\nu=\partial z.
	\end{equation}
	
	\noindent Using (\ref{eq43x}), (\ref{eq44y}) and (\ref{eq45z}), we  have $Rad(TM)=span\{{l_{1}}\},  \; \phi_{o}(Rad (TM)) = span \{{l_{2}}\}$ and
	$Rad(TM) \cup \phi_{o} (Rad(TM))=\{0\}$. Also, $\phi_{o}(l_{3})=-l_{4}$ and $\phi_{o}(l_{4})=l_{3}$, which implies that $\phi_{o} D_{o} = D_{o}$, where $D_{o}=\{l_{3},l_{4}\}$.\\
	
	\noindent Further, the screen transversal vector bundle and lightlike transversal vector bundle are given by\\
	$S(TM^{\perp})=span \{W=e^{-z}(\partial x_{2} +\dfrac{y^{2}}{x^{2}}\partial y_{2})$
	where $\phi_{o}(W)=-l_{5}$ and\\ 
	$ltr(TM) = span  \{ N=\dfrac{e^{-z}}{2}(-\partial x_{1}+\partial y_{4})$ where $\phi_{o}(N)=\dfrac{1}{2}l_{6}$, respectively which gives $D^{'}= span\{\phi_{o} N, \phi_{o} W\}$.\\
	
	\noindent Therefore, $M$ becomes a contact $CR$-lightlike submanifold of the indefinite Kenmotsu statistical manifold $(\mathbb{R}^9_2,\widehat{\bar{\nabla}} + K,\phi_{o},\nu,\eta,\bar{g})$.
\end{example}

\begin{theorem}
	For a contact $CR$-lightlike submanifold $M$ of an indefinite Kenmotsu statistical manifold $\bar{M}$, $D\perp\{\nu\}$ is a totally geodesic foliation iff
	\[
	\bar{g}({h^{*}}^{l}(X,\phi Y),\xi)+\bar{g}(h^{l}(X,\phi Y),\xi)=0
	\] 
	\[
	\bar{g}({h^{*}}^{s}(X,Y),\phi W)+\bar{g}(h^{s}(X,Y),\phi W)=0
	\]
	
	\noindent for all $X,Y \in \Gamma(D\perp\{\nu\})$ and $W \in \Gamma(\phi L_{1})$.
\end{theorem}

\begin{proof}
	$D\perp\{\nu\}$ defines a totally geodesic foliation iff
	\[
	\bar{g}(\widehat{\nabla}_{X}Y,\phi\xi)=\bar{g}(\widehat{\nabla}_{X}Y,W)=0\]
	for $X,Y \in \Gamma(D\perp\{\nu\})$ and $W \in \Gamma(\phi L_{1})$. Now from the concept of indefinite Kenmotsu statistical manifold and the equations (\ref{eq38}) and (\ref{eq32}), we have\\
	\noindent	$\bar{g}(\widehat{\nabla}_{X}Y,\phi\xi)= \bar{g}(\widehat{\bar{\nabla}}_{X}Y,\phi\xi)\\
	=\frac{1}{2}[-\bar{g}(\phi\bar{\nabla}_{X}Y,\xi)-\bar{g}(\phi\bar{\nabla}_{X}^{*}Y,\xi)]\\
	=-\frac{1}{2}[\bar{g}(\bar{\nabla}^{*}_{X}\phi Y,\xi)+\bar{g}(\bar{\nabla}_{X}\phi Y,\xi)]\\
	=-\frac{1}{2}[\bar{g}({h^{*}}^{l}(X,\phi Y),\xi)+\bar{g}(h^{l}(X,\phi Y),\xi)]$\\
	
	\noindent And	$\bar{g}(\widehat{\nabla}_{X}\phi Y,W)=\frac{1}{2}[\bar{g}(\bar{\nabla}_{X}\phi Y, W)+\bar{g}(\bar{\nabla}^{*}_{X}\phi Y, W)]\\
	=\frac{1}{2}[\bar{g}(\phi\bar{\nabla}^{*}_{X}Y+\eta(Y)\phi X+\bar{g}(\phi X,Y)\nu,W)+\bar{g}(\phi\bar{\nabla}_{X}Y+\eta(Y)\phi X+\bar{g}(\phi X, Y)\nu,W)]\\
	=-\frac{1}{2}[\bar{g}(\bar{\nabla}^{*}_{X}Y,\phi W)+\bar{g}(\bar{\nabla}_{X}Y,\phi W)]\\
	=-\frac{1}{2}[\bar{g}({h^{*}}^{s}(X,Y),\phi W)+\bar{g}(h^{s}(X,Y),\phi W)]$.\\
	
	\noindent	which proves the assertion. 
\end{proof}

\begin{theorem}
	If $M$ is a contact $CR$-lightlike submanifold of an indefinite Kenmotsu statistical manifold $(\bar{M},\bar{\nabla},\bar{\nabla}^{*},\bar{g},\phi,\nu)$, then $D$ is integrable if and only if
	\[
	{h^{*}}^{s}(X,\phi Y)+h^{s}(X,\phi Y)={h^{*}}^{s}(Y,\phi X)+h^{s}(Y,\phi X) \quad \forall \; X,Y \in \Gamma(D).
	\]
\end{theorem}

\begin{proof}
	$D$ is integrable if and only if 
	\[
	\bar{g}([X,Y],\nu)=0, \quad \bar{g}([X,Y],W)=0 \quad \forall \; X,Y \in \Gamma(D), \; W \in \Gamma (D^{\prime}).
	\]
	Since $\bar{M}$ is an indefinite Kenmotsu statistical manifold and (\ref{eq37}), (\ref{eq38}), (\ref{eq31}), (\ref{eq42b}) hold, therefore, we have $\bar{g}([X,Y],\nu)=0$\\
	
	\noindent	Also $\bar{g}([X,Y],W)=\frac{1}{2}[\bar{g}(\bar{\nabla}_{X}Y,W)+\bar{g}(\bar{\nabla}^{*}_{X}Y,W)-\bar{g}(\bar{\nabla}_{Y}X,W)-\bar{g}(\bar{\nabla}^{*}_{Y}X,W)]\\
	=\frac{1}{2}[\bar{g}(\bar{\nabla}^{*}_{X}\phi Y+\bar{g}(\phi X,Y)\nu,\phi W)+\bar{g}(\bar{\nabla}_{X}\phi Y+\bar{g}(\phi X,Y)\nu,\phi W)\\-\bar{g}(\bar{\nabla}^{*}_{Y}\phi X+\bar{g}(\phi Y,X)\nu,\phi W)-\bar{g}(\bar{\nabla}_{Y}\phi X+\bar{g}(\phi Y,X)\nu,\phi W)]\\
	=\frac{1}{2}[\bar{g}({h^{*}}^{s}(X,\phi Y),\phi W)+\bar{g}(h^{s}(X,\phi Y),\phi W)-\bar{g}({h^{*}}^{s}(Y,\phi X),\phi W)-\bar{g}(h^{s}(Y,\phi X),\phi W)].$\\
	
	\noindent Hence, we obtain the desired result using hypothesis. 
	
\end{proof}

\begin{theorem}
	If $\bar{M}$ is an indefinite Kenmotsu statistical manifold and $M$ is a contact $CR$-lightlike submanifold of $\bar{M}$, then $D^{\prime}$ does not define a totally geodesic foliation.
\end{theorem}

\begin{proof}
	$D^{\prime}$ defines a toatlly geodesic foliation iff 
	\begin{equation}\label{eq76}
		\bar{g}(\widehat{\nabla}_{Z}W,N)=\bar{g}(\widehat{\nabla}_{Z}W,\phi N)=\bar{g}(\widehat{\nabla}_{Z}W,X)=\bar{g}(\widehat{\nabla}_{Z}W,\nu)=0
	\end{equation}
	
	\noindent	where $Z,W \in \Gamma(D^{\prime})$, $N\in \Gamma(ltr(TM))$ and $X \in \Gamma(D_{0})$.\\
	
	\noindent	Therefore, using (\ref{eq37}), (\ref{eq68}), we obtain\\
	
	\noindent	$\bar{g}(\widehat{\nabla}_{Z}W,\nu)=\bar{g}(\widehat{\bar{\nabla}}_{Z}W,\nu)\\
	=-\frac{1}{2}[\bar{g}(W,-Z+\eta(Z)\nu+\eta(\bar{\nabla}_{Z}^{*}\nu)\nu)+\bar{g}(W,-Z+\eta(Z)\nu+\eta(\bar{\nabla}_{Z}\nu)\nu)]\\
	=\bar{g}(W,Z)$	\\
	
	\noindent And $\bar{g}(\widehat{\nabla}_{Z}W,X)=\frac{1}{2}[\bar{g}(\phi\bar{\nabla}_{Z}W,\phi X)+\bar{g}(\phi\bar{\nabla}^{*}_{Z}W,\phi X)]\\
	=\frac{1}{2}[\bar{g}(\bar{\nabla}^{*}_{Z}\phi W-\eta(W)\eta(\phi Z)+\bar{g}(\phi Z,W)\nu,\phi X)+\bar{g}(\bar{\nabla}_{Z}\phi W-\eta(W)\eta(\phi Z)\\+\bar{g}(\phi Z,W),\phi X)]\\
	=\frac{1}{2}[\bar{g}(-A^{*}_{\phi W}Z,\phi X)+\bar{g}(-A_{\phi W}Z,\phi X)]$\\
	
	\noindent Since $\bar{g}(W,Z)\ne0$ and $\bar{g}(-A^{*}_{\phi W}Z,\phi X)+\bar{g}(-A_{\phi W}Z,\phi X)\ne0$, therefore $D^{'}$ does not define a totally geodesic foliation using (\ref{eq76}).
	
\end{proof}

\begin{theorem}
	For a contact $CR$-lightlike submanifold $M$ of an indefinite Kenmotsu statistical manifold $(\bar{M},\bar{\nabla},\bar{\nabla}^{*},\bar{g},\phi,\nu)$, $D\perp\{\nu\}$ is integrable if and only if
	\[
	h(X,\phi Y)+h^{*}(X,\phi Y)=h(Y,\phi X)+h^{*}(Y,\phi X)
	\]
	
	\noindent 	for all $X,Y \in D\perp\{\nu\}$.
	
\end{theorem}	

\begin{proof}
	Since, $\bar{M}$ is an indefinite Kenmotsu statistical manifold and (\ref{eq38}), (\ref{eq31}), (\ref{eq45b}), (\ref{eq46b}) hold, therefore, on considering
	\[h(X,\phi Y)=\phi\bar{\nabla}_{X}^{*}Y+\eta(Y)\phi X-\bar{g}(\phi X,Y)\nu-\nabla_{X}\phi Y\]
	
	\noindent	we obtain, $w(\nabla_{X}^{*}Y)=h(X,\phi Y)-C({h^{s}}^{*}(X,Y)) \quad \forall \; X,Y \in D\perp\{\nu\}$.\\
	
	\noindent	Similarly, we derive $w(\nabla_{X}Y)=h^{*}(X,\phi Y)-C(h^{s}(X,Y))$.\\
	
	\noindent On interchanging $X$ and $Y$ in the above equations, we get\\
	
	\noindent $ w(\nabla_{Y}^{*}X)=h(Y,\phi X)-C({h^{s}}^{*}(Y,X))$
	and $ w(\nabla_{Y}X)=h^{*}(Y\phi X)-C(h^{s}(Y,X)).$\\
	
	\noindent	Hence, we have $h(X,\phi Y)+h^{*}(X,\phi Y)-h(Y,\phi X)-h^{*}(Y,\phi X)= 2w([X,Y])$.\\
	
	\noindent	Thus, the result follows from the integrability of $D\perp\{\nu\}$.
	
\end{proof}

\section{Geodesic contact SCR-lightlike submanifolds}

\begin{definition}
	\cite{ref9} Let $(M, g, S(TM), S(TM^{\perp}))$ be a lightlike submanifold tangent to the structure vector field $\nu$, immersed in an indefinite Kenmotsu manifold $(\bar{M},\bar{g})$. Then $M$ is a contact $SCR$-lightlike submanifold of $\bar{M}$ if the following conditions hold:
\end{definition}

\begin{enumerate}
	\item There exist real non-empty distributions $D$ and $D^{\perp}$ such that
	\[
	S(TM) = D \perp D^{\perp} \perp \{\nu\}, 
	\]
	\[
	\phi (D^{\perp}) \subset (S(TM^{\perp})), \hspace{.2in} D \cap D^{\perp} = \{0\}	
	\]
	
	where $D^{\perp}$ is orthogonally complementary to $D \perp \{\nu\}$ in $S(TM)$.\\
	\item The distributions $D$ and $Rad(TM)$ are invariant with respect to $\phi$. It follows that $ltr(TM)$ is also invariant with respect to $\phi$. Hence we have the decomposition:
	\begin{equation}\label{eq48}
		TM = \bar{D} \perp D^{\perp} \perp \{\nu\}, \hspace{.2in} \bar{D}= D \perp Rad(TM).
	\end{equation}  
\end{enumerate}

\noindent For any $X\in\Gamma(TM)$ and $W\in \Gamma(S(TM^{\perp}))$, we put
\begin{equation}\label{eq47}
	\phi X=P^{\prime}X+F^{\prime}X, \qquad \phi W=B^{\prime}W+C^{\prime}W,
\end{equation}
where $P^{\prime}X\in \Gamma(\bar{D}), \; F^{\prime}X\in \Gamma(\phi D^{\perp}), \; B^{\prime}W\in\Gamma(D^{\perp}), \; C^{\prime}W\in \Gamma(\mu^{\prime})$ and $\mu^{\prime}$ denotes the orthogonal complement to $\phi(D^{\perp})$ in $S(TM^{\perp})$.\\

\noindent Now, an example elaborating the structure of an $SCR$-lightlike submanifold of an indefinite Kenmotsu statistical manifold following \cite{ref9} is presented as:

\begin{example}
	Consider the indefinite Kenmotsu statistical manifold $\bar{M}=(\mathbb{R}^9_2,\bar{\nabla} = \widehat{\bar{\nabla}} + K,\phi_{o},\nu,\eta,\bar{g})$ with the structure defined as in example (\ref{example}).\\
	
	\noindent	Now let $M$ be a submanifold of $\bar{M}$ defined by
	\[
	x^1=x^2, \qquad y^1=y^2, \qquad x^4=\sqrt{1-{(y^4)}^2}, \qquad y^4\ne \pm 1.
	\]
	
	\noindent	Then, the local frame of $TM$ corresponding to the submanifold of indefinite Kenmotsu statistical manifold is constructed as under:
	\begin{equation}
		l_1=e^{-z}(\partial x_1+\partial x_2),  \qquad l_2=e^{-z}(\partial y_1+\partial y_2),
	\end{equation} 
	\begin{equation}
		l_3=e^{-z}\partial x_3, \qquad l_4=e^{-z}\partial y_3,
	\end{equation}
	\begin{equation}
		l_5=e^{-z}(-y^4\partial x_4+x^4\partial y_4), \qquad l_6=\nu=\partial z.
	\end{equation}
	
	\noindent Then, $Rad(TM)=span\{l_1,l_2\}$ as $\phi_o(l_1)=-l_2$ and $D=span\{l_3,l_4\}$ as $\phi_o(l_3)=-l_4$.\\
	
	\noindent	Also, $S(TM^\perp)=span\{W=e^{-z}(x^4\partial x_4+y^4\partial y_4)\}$ where $\phi_o(W)=-l_5$\\
	and $ltr(TM)=span\{N_1=\dfrac{e^{-z}}{2}(-\partial x_1+\partial x_2),\; N_2=\dfrac{e^{-z}}{2}(-\partial y_1+\partial y_2)\}$ where $\phi_o(N_2)=N_1$. This shows that $Rad(TM)$, $D$ and $ltr(TM)$ are invariant with respect to $\phi_{o}$.\\
	
	\noindent	Therefore, $M$ is said to be a contact $SCR$-lightlike submanifold of the indefinite Kenmotsu statistical manifold $\bar{M}$.
\end{example}

\begin{theorem}
	For a contact $SCR$-lightlike submanifold $M$ of an indefinite Kenmotsu statistical manifold $\bar{M}$, the induced connections $\nabla \; (respectively \; \nabla^{*})$ are metric connections iff 
	${A^*}^{\prime}_{\xi}X+A^{\prime}_{\xi}X$ and ${h^*}^{s}(X,\xi)+h^{s}(X,\xi)$ have no components in $D$ and $\phi(D^{\perp})$ respectively, for all $ X \in \Gamma(TM), \; \xi \in \Gamma(Rad(TM))$.
\end{theorem}

\begin{proof}
	Replacing $Y$ by $\xi$ in (\ref{eq67}) and simplifying further, we derive $\bar{\nabla}_{X}\phi \xi=\phi \bar{\nabla}^{*}_{X}\xi$ and $\bar{\nabla}^{*}_{X}\phi \xi=\phi \bar{\nabla}_{X}\xi$, for all $X \in \Gamma(TM), \; \xi \in \Gamma(Rad(TM))$.\\
	
	\noindent Now using (\ref{eq38}), (\ref{eq312}) and (\ref{eq47}), we obtain\\
	
	\noindent $\nabla_{X}\phi\xi+h^{l}(X,\phi\xi)+h^{s}(X,\phi\xi)=\phi(-{A^*}^{\prime}_{\xi}X+{\nabla^*}{\prime}_X^{\perp}\xi)+\phi({h^{*}}^{l}(X,\xi))+\phi({h^{*}}^{s}(X,\xi))$.\\
	
	\noindent	Therefore, taking the tangential parts, we have
	\begin{equation}\label{A}
		\nabla_{X}\phi\xi=-P^{\prime}({A^*}^{\prime}_{\xi}X)+\phi(\nabla^*{\prime}_{X}^{^{\perp}}\xi)+B^{\prime}({h^*}^{s}(X,\xi)). 
	\end{equation}
	
	\noindent	Similarly, we derive
	\begin{equation}\label{B}
		\nabla_{X}^*\phi\xi=-P^{\prime}(A^{\prime}_{\xi}X)+\phi({\nabla^{\prime}_{X}}^{\perp}\xi)+B^{\prime}(h^{s}(X,\xi)).
	\end{equation}
	
	\noindent	Since the induced connection is a metric connection iff $Rad(TM)$ is parallel with respect to $\nabla$\;(respectively\;$\nabla^*)$, so by letting $Rad(TM)$ to be parallel, we have $\nabla_{X}\phi\xi=0$ and $\nabla^*_{X}\phi\xi=0$.\\
	
	\noindent	From (\ref{A}) and (\ref{B}), we get \\
	
	\noindent $\bar{g}(P^{\prime}({A^*}^{\prime}_{\xi}X+A^{\prime}_{\xi}X), \phi Z)=0$ and $\bar{g}(B^{\prime}({h^*}^{s}(X,\xi)+h^{s}(X,\xi)),\phi Z)=0$\\
	
	\noindent	where $P^{\prime}({A^*}^{\prime}_{\xi}X),P^{\prime}(A^{\prime}_{\xi}X)\in \Gamma(\bar{D}), \; B^{\prime}({h^*}^{s}(X,\xi)),B^{\prime}(h^{s}(X,\xi))\in\Gamma(D^{\perp})$ and $\phi Z\in \Gamma(D^{\perp})$.\\ 
	
	\noindent This shows that ${A^*}^{\prime}_{\xi}X+A^{\prime}_{\xi}X$ has no components in $D$ and ${h^*}^{s}(X,\xi)+h^{s}(X,\xi)$ has no components in $\phi(D^{\perp})$.\\
	
	\noindent	Conversely, from (\ref{A}) and (\ref{B}), we have
	\[
	\nabla_{X}\phi\xi,\; \nabla^*_{X}\phi\xi \in \Gamma(Rad(TM)).\]
	Since $Rad(TM)$ is parallel, therefore the induced connections $\nabla$ (respectively  $\nabla^{*})$ are metric connections. 
	
\end{proof}

\begin{definition}
	Let $\bar{M}$ be an indefinite Kenmotsu statistical manifold and $M$ be a contact $SCR$-lightlike submanifold of $\bar{M}$. Then, $M$ is said to be $D^{\perp}$-totally geodesic contact $SCR$-lightlike submanifold with respect to $\bar{\nabla}$ (respectively $\bar{\nabla}^{*}$) if $h(X,Y)=0$ (respectively $h^{*}(X,Y)=0$) $\forall \;  X, Y \in D^{\perp}$.
\end{definition}

\begin{theorem}
	For a contact $SCR$-lightlike submanifold $M$ of an indefinite Kenmotsu statistical manifold $\bar{M}$, $M$ is $D^{\perp}$-totally geodesic if and only if
	\[
	\bar{g}({D^{*}}^{l}(X,F^{\prime}Y), \phi\xi)+\bar{g}(D^{l}(X,F^{\prime}Y),\phi\xi)=0
	\]
	\[
	\bar{g}({{\nabla}^{*}}^{s}_{X}\phi Y,C^{\prime}W)+\bar{g}(\nabla^{s}_{X}\phi Y,C^{\prime}W)=\bar{g}(A^{*}_{\phi Y}X,B^{\prime}W)+\bar{g}(A_{\phi Y}X,B^{\prime}W)
	\]		
	for all $X, Y \in \Gamma(D^{\perp})$.
	
\end{theorem}

\begin{proof}
	Using (\ref{eq31}) and (\ref{eq47}) for a contact $SCR$-lightlike submanifold $M$ of the indefinite Kenmotsu statistical manifold $\bar{M}$, we have\\
	
	\noindent $\bar{g}(h^{l}(X,Y),\xi) + \bar{g}({{h}^{*}}^{l}(X,Y),\xi)=\bar{g}(\bar{\nabla}_{X}Y, \xi)+ \bar{g}(\bar{\nabla}^{*}_{X}Y, \xi)\\
	=\bar{g}(\bar{\nabla}^{*}_{X}\phi Y-\eta(Y)\phi X+\bar{g}(\phi X,Y)\nu,\phi\xi)+\bar{g}(\bar{\nabla}_{X}\phi Y-\eta(Y)\phi X+\bar{g}(\phi X,Y)\nu,\phi\xi)$\\
	
	\noindent which gives 
	\begin{equation}\label{eqG}
		\bar{g}(h^{l}(X,Y),\xi) + \bar{g}({{h}^{*}}^{l}(X,Y),\xi)	=\bar{g}({D^{*}}^{l}(X,F^{\prime}Y), \phi\xi)+\bar{g}(D^{l}(X,F^{\prime}Y),\phi\xi)
	\end{equation}
	
	\noindent	And $\bar{g}(h^{l}(X,Y),W) + \bar{g}({{h}^{*}}^{l}(X,Y),W)=\bar{g}(\bar{\nabla}_{X}Y, W)+ \bar{g}(\bar{\nabla}^{*}_{X}Y, W)\\
	=\bar{g}(\bar{\nabla}^{*}_{X}\phi Y,\phi W)+\bar{g}(\bar{\nabla}_{X}\phi Y,\phi W)$\\
	
	\noindent Then 
	\begin{equation}\label{eqH}
		\begin{split}
			\bar{g}(h^{l}(X,Y),W) + \bar{g}({{h}^{*}}^{l}(X,Y),W)
			=\bar{g}(-A^{*}_{\phi Y}X,B^{\prime}W)+\bar{g}({{\nabla}^{*}}^{s}_{X}\phi Y,C^{\prime}W)\\+\bar{g}(-A_{\phi Y}X,B^{\prime}W)+\bar{g}(\nabla^{s}_{X}\phi Y,C^{\prime}W).
		\end{split}
	\end{equation}
	
	\noindent	If $M$ is $D^{\perp}$-totally geodesic, then
	\[
	\bar{g}(h^{l}(X,Y),\xi)=0, \qquad \bar{g}({{h}^{*}}^{l}(X,Y),\xi)=0\]
	\[
	\bar{g}(h^{s}(X,Y),W)=0 \qquad \bar{g}({{h}^{*}}^{s}(X,Y),W)=0 \]
	for all $X, Y \in \Gamma(D^{\perp})$.\\
	
	\noindent	Therefore, the hypothesis alongwith (\ref{eqG}) and (\ref{eqH}) leads to the desired result.\\
	
	\noindent	Conversely, using (\ref{eqG}) and (\ref{eqH}) for all $X, Y \in \Gamma(D^{\perp})$, $M$ becomes $D^{\perp}$-totally geodesic. 
	
\end{proof}

\begin{theorem} 
	Let $(\bar{M},\bar{\nabla},\bar{\nabla}^{*},\bar{g},\phi,\nu)$ be an indefinite Kenmotsu statistical manifold and $M$ be a contact $SCR$-lightlike submanifold of $\bar{M}$. Then $\bar{D}\perp\{\nu\}$ defines a totally geodesic foliation in $M$ iff
	\[
	\bar{g}({h^{*}}^{s}(X, \phi Y) + h^{s}(X,\phi Y),\phi Z)=0
	\]
	
	\noindent for all $X,Y \in \Gamma(\bar{D}\perp\{\nu\})$.
	
\end{theorem}

\begin{proof}
	As per the concept of indefinite Kenmotsu statistical manifold  and from (\ref{eq31}), for all $X,Y \in \Gamma(\bar{D}\perp\{\nu\})$, we have\\
	
	\noindent	$\bar{g}(\widehat{\nabla}_{X}Y, Z)= \bar{g}(\widehat{\bar{\nabla}}_{X}Y,Z)\\
	=\frac{1}{2}[\bar{g}(\phi\bar{\nabla}_{X}Y, \phi Z)+\bar{g}(\phi \bar{\nabla}^{*}_{X}Y, \phi Z)]\\
	=\frac{1}{2}[\bar{g}(\bar{\nabla}^{*}_{X}\phi Y-\eta(Y)\phi X+\bar{g}(\phi X,Y)\nu,\phi Z)+ \bar{g}(\bar{\nabla}_{X}\phi Y-\eta(Y)\phi X+\bar{g}(\phi X,Y)\nu,\phi Z)]\\
	=\frac{1}{2}[\bar{g}({h^{*}}^{s}(X, \phi Y),\phi Z)+\bar{g}(h^{s}(X,\phi Y),\phi Z)]\\$
	
	\noindent Since, $\bar{D}\perp\{\nu\}$ defines a totally geodesic foliation in $M$ iff
	\[
	\bar{g}(\widehat{\nabla}_{X}Y, Z)=0 \quad \forall \; X, Y \in \Gamma(\bar{D}\perp\{\nu\})
	\]
	Hence the proof. 	
	
\end{proof}	 

\begin{theorem}
	Let $M$ be a contact $SCR$-lightlike submanifold of an indefinite Kenmotsu statistical manifold $\bar{M}$. Then $D^{\perp}$ and $\bar{D}$ do not define a geodesic foliation on $M$.
\end{theorem}

\begin{proof}
	$D^{\perp}$ defines a totally geodesic foliation in a contact $SCR$-lightlike submanifold $M$ iff
	\[
	\bar{g}(\widehat{\nabla}_{X}Y, Z)=\bar{g}(\widehat{\nabla}_{X}Y, \nu)=\bar{g}(\widehat{\nabla}_{X}Y,N)=0
	\]
	$\forall \; X,Y \in \Gamma(D^{\perp}), \; Z \in \Gamma(\bar{D}), 
	\; N \in \Gamma(ltr(TM)).$\\
	
	\noindent	Now using (\ref{eq37}) and (\ref{eq68}),\\
	
	\noindent	$\bar{g}(\widehat{\nabla}_{X}Y, \nu)=\frac{1}{2}[-\bar{g}(Y, -X+(\eta(X)+\mu(X))\nu)-\bar{g}(Y,-X+(\eta(X)+\mu(X))\nu)]\\
	=\bar{g}(X,Y)$ \\ 
	
	\noindent	For the choice of non-null vectors $X,Y \in \Gamma(D^{\perp})$, we have $\bar{g}(X,Y) \ne 0$. Hence, $D^{\perp}$ does not define a totally geodesic foliation on $M$.\\ 
	
	\noindent	Also, using (\ref{eq31}), (\ref{eq67}) and (\ref{eq48}), we have\\
	
	\noindent	$\bar{g}(\widehat{\nabla}_{X}Y, Z)=\frac{1}{2}[\bar{g}(\bar{\nabla}_{X}Y,Z)+\bar{g}(\bar{\nabla}^{*}_{X}Y,Z)]\\
	=\frac{1}{2}[-\bar{g}(\phi Y, \phi \bar{\nabla}^{*}_{X}Z)-\bar{g}(\phi Y,\phi \bar{\nabla}_{X}Z)]\\
	=-\frac{1}{2}[\bar{g}(\phi Y, \bar{\nabla}_{X}\phi Z-\eta(Z)\phi X+\bar{g}(\phi X,Z)\nu)+\bar{g}(\phi Y, \bar{\nabla}^{*}_{X}\phi Z-\eta(Z)\phi X+\bar{g}(\phi X,Z)\nu)]\\
	= -\frac{1}{2}[\bar{g}(\phi Y,\nabla_{X}\phi Z)+\bar{g}(\phi Y,\nabla_{X}^{*}\phi Z)] \\
	\ne 0 \qquad \forall \; X, Y \in \Gamma(\bar{D}), \; Z \in \Gamma(D^{\perp}).$\\
	
	\noindent	which implies that $\bar{D}$ does not define a totally geodesic foliation on $M$. 
	
\end{proof}

\begin{definition}
	A contact $SCR$-lightlike submanifold of indefinite Kenmotsu statistical manifold $\bar{M}$ is said to be mixed totally geodesic contact $SCR$-lightlike submanifold with respect to $\bar{\nabla}$ (respectively $\bar{\nabla}^{*}$) if $h(X,Y)=0$ (respectively $h^{*}(X,Y)=0$) $\forall \;  X \in \Gamma(\bar{D}\perp \{\nu\})$ and $\forall \; Y \in D^{\perp}$.	
\end{definition}

\begin{theorem}
	For an indefinite Kenmotsu statistical manifold $\bar{M}$, let $M$ be a contact $SCR$-lightlike submanifold. Then, $M$ is mixed totally geodesic iff
	\[
	\bar{g}({D^{*}}^{l}(X,F^{\prime}Y), \phi\xi)+\bar{g}(D^{l}(X,F^{\prime}Y),\phi\xi)=0
	\]
	\[
	\bar{g}({{\nabla}^{*}}^{s}_{X}\phi Y,C^{\prime}W)+\bar{g}(\nabla^{s}_{X}\phi Y,C^{\prime}W)=\bar{g}(A^{*}_{\phi Y}X,B^{\prime}W)+\bar{g}(A_{\phi Y}X,B^{\prime}W)
	\]		
	for all $X \in \Gamma(\bar{D}\perp \{\nu\})$ and $\forall \; Y \in \Gamma(D^{\perp})$. 
	
\end{theorem}

\begin{proof}
	
	Let $M$ be mixed totally geodesic contact $SCR$-lightlike submanifold of $\bar{M}$. Then, we have 
	\begin{equation}\label{I}
		\begin{split}
			\bar{g}(h^{l}(X,Y),\xi)=0, \qquad \bar{g}({{h}^{*}}^{l}(X,Y),\xi)=0\\
			\bar{g}(h^{s}(X,Y),W)=0 \qquad \bar{g}({{h}^{*}}^{s}(X,Y),W)=0 
		\end{split}
	\end{equation}
	
	\noindent	for all $X \in \Gamma(\bar{D}\perp \{\nu\})$ and $\forall \; Y \in \Gamma(D^{\perp})$.\\
	
	\noindent	Since, $\bar{M}$ is the indefinite Kenmotsu statistical manifold, therefore using (\ref{eq31}) and (\ref{eq47}) for all $X \in \Gamma(\bar{D}\perp \{\nu\})$ and $\forall \; Y \in \Gamma(D^{\perp})$, we have\\
	
	\noindent $\bar{g}(h^{l}(X,Y),\xi) + \bar{g}({{h}^{*}}^{l}(X,Y),\xi)=\bar{g}(\phi\bar{\nabla}_{X}Y, \phi\xi)+ \bar{g}(\phi\bar{\nabla}^{*}_{X}Y, \phi\xi)+\eta(\bar{\nabla}_{X}Y)\eta(\xi)\\
	\hspace{4in}+\eta(\bar{\nabla}^{*}_{X}Y)\eta(\xi)=\bar{g}(\bar{\nabla}^{*}_{X}\phi Y, \phi\xi)+ \bar{g}(\bar{\nabla}_{X}\phi Y, \phi\xi)$\\
	
	\noindent that is,  $\bar{g}(h^{l}(X,Y),\xi) + \bar{g}({{h}^{*}}^{l}(X,Y),\xi)=\bar{g}({D^{*}}^{l}(X,F^{\prime}Y), \phi\xi)+\bar{g}(D^{l}(X,F^{\prime}Y),\phi\xi)$\\
	
	\noindent	and $\bar{g}(h^{l}(X,Y),W) + \bar{g}({{h}^{*}}^{l}(X,Y),W)=\bar{g}(\phi\bar{\nabla}_{X}Y, \phi W)+ \bar{g}(\phi \bar{\nabla}^{*}_{X}Y, \phi W)\\
	=\bar{g}(\bar{\nabla}^{*}_{X}\phi Y-\eta(Y)\phi X+\bar{g}(\phi X,Y)\nu,\phi W)+\bar{g}(\bar{\nabla}_{X}\phi Y-\eta(Y)\phi X+\bar{g}(\phi X,Y)\nu,\phi W)$\\
	
	\noindent that is, $\bar{g}(h^{l}(X,Y),W) + \bar{g}({{h}^{*}}^{l}(X,Y),W)
	=\bar{g}(-A^{*}_{\phi Y}X,B^{\prime}W)+\bar{g}({{\nabla}^{*}}^{s}_{X}\phi Y,C^{\prime}W)$\\
	
	\hspace{2.6in} $+\bar{g}(-A_{\phi Y}X,B^{\prime}W)+\bar{g}(\nabla^{s}_{X}\phi Y,C^{\prime}W).$\\
	
	\noindent So, the required result is obtained  following (\ref{I}).  
	
\end{proof}


{\bf Acknowledgement.} The authors wish to thank the anonymous referee for his/her valuable suggestions that helped to improve the manuscript.

\bibliographystyle{amsplain}

\end{document}